\documentclass[12pt,a4paper]{article}
\usepackage{amssymb,latexsym,amsmath}
\usepackage{float}              

\usepackage{epsfig}

\usepackage{subfig}
\usepackage{graphicx}

\input epsf

\newtheorem{prop}{Proposition}[section]
\newtheorem{theo}[prop]{Theorem}

\newtheorem{rem}[prop]{Remark}

\newenvironment{proof}
 {\begin{trivlist} \item[\hskip \labelsep {\bf Proof}\hspace*{3 mm}]}
 {\hfill$\Box$\end{trivlist}}

\begin{document}

\title{Focal set of curves in the
Minkowski space near lightlike points}
\author{Ana Claudia Nabarro\footnote{Supported by  FAPESP grant
 2013/02794-4} \, and\, Andrea de Jesus Sacramento \footnote{Supported by  FAPESP grant
 2010/20301-7} }
\date{}

\maketitle
\begin{abstract}
 We study the geometry of curves in the Minkowski space and in the de Sitter space, specially at points where the tangent direction is lightlike (i.e. has length zero) called lightlike points of the curve. We define the focal sets of these curves and study the metric structure of them. At the lightlike points, the focal set is not defined. We use  singularity theory techniques to carry out our study and investigate the focal set near lightlike points.
\end{abstract}

\renewcommand{\thefootnote}{\fnsymbol{footnote}}
\footnote[0]{2010 Mathematics Subject classification 58K05, 53D10, 53B30.}


\section{Introduction}\label{sec:intro}

The study of submanifolds in Minkowski space is of interest in relativity theory. We believe that it is important to study the geometry of submanifolds in the Minkowski space with the induced metric degenerating at some points on the submanifolds. For example, any closed (compact without boundary) surfaces in the Minkowski 3-space has an non-empty locus of points where the metric is degenerate. (We observe that there are various studies  in geometry on such submanifolds.  For example, in \cite{stellergauss} a Gauss-Bonnet type theorem is proven, and in \cite{pelletier}, the problem of how to extend the Levi-Civita connection to the locus of degeneracy of the metric is considered). For this reason, some authors started to investigate  the geometry of such submanifolds using the singularity theory. The first step was to study the cases of curves in the Minkowski plane \cite{Saloomtari} and of surfaces in the Minkowski 3-space \cite{farid}. In \cite{farid}, the authors studied the caustics of surfaces in the Minkowski 3-space. Although the focal set of the surface is not defined at points where the metric is degenerate, the caustic is. The properties of the induced metric on the caustic are studied in \cite{farid}. Submanifolds in pseudo-spheres of the Minkowski space are also studied in several papers.
In \cite{izumiyapeisanosingularities}, Izumiya-Pei-Sano defined the hyperbolic Gauss indicatrix of a hypersurface in the Minkowski space model of the hyperbolic space. The work in \cite{izumiyapeisanosingularities} set the foundations of applications of singularity theory for the extrinsic geometry of submanifolds in the hyperbolic space. The extrinsic geometry of spacelike or timelike submanifolds in other pseudo-spheres  of the Minkowski-space are investigated in subsequent papers.


In this paper, we study the geometry of curves in Minkowski $3$-space and  in pseudo-spheres $S^2_1$ and $S^3_1$.  In order to do this, we study the families of  distance square functions on the curves.  We study their focal sets and the bifurcation set of the family of the distance square functions on these curves in order to investigate what happens near the lightlike points.

The paper is organised as follows. Section \ref{sec:pre} addresses some preliminary results and notions that are used in the paper. We define an open and dense set of curves, such that
the lightlike points of a curve are isolated. Besides, given a curve in this set, passing by  lightlike points, the curve changes from spacelike to timelike.

 We consider in $\S\ref{sec:focalst}$ spacelike and timelike curves $\gamma$ using the Frenet-Serret formulae. These formulae and the family of distance squared functions on $\gamma$ are the main tools in this section.  Here  we study the geometry and metric structure of the focal set of $\gamma$.

In Section \ref{sec:fslig}, we study the bifurcation set of the family of distance squared functions on $\gamma$ in the neighborhood of lightlike points of the curve. In this case, we cannot parametrise the curve by arc length, therefore we cannot  use the Frenet-Serret formulae as in $\S\ref{sec:focalst}$.

 In $\S\ref{sec:focals21}$ and $\S\ref{sec:focals31}$, we consider curves, the focal sets and the bifurcation sets  in the pseudo-spheres $S^2_1$ and $S^3_1$. We study the metric structures of these sets  locally at lightlike points of $\gamma$.

\section{Preliminaries}\label{sec:pre}
 The \emph{Minkowski space} $\mathbb{R}^{n+1}_1$ is the vector space $\mathbb{R}^{n+1}$ endowed with the pseudo-scalar product $\langle x,y\rangle=-x_1y_1 + x_2y_2 +\ldots+ x_{n+1}y_{n+1},$ for any $x=(x_1,x_2,\ldots,x_{n+1})$ and $y=(y_1,y_2,\ldots,y_{n+1})$ in $\mathbb{R}^{n+1}$. We say that a non-zero vector $x\in\mathbb{R}^{n+1}_1$ is \emph{spacelike} if $\langle x,x\rangle > 0$, \emph{lightlike} if $\langle x,x\rangle = 0$ and
 \emph{ timelike} if $\langle x,x\rangle < 0$.
The norm of a vector $x\in\mathbb{R}^{n+1}_1$ is defined by $\parallel x\parallel=\sqrt{\mid\langle x,x\rangle\mid}$. This is an example of the Lorentzian metric. In $\mathbb{R}^{3}_1$, the pseudo vector product of $x=(x_1,x_2,x_3)$ and $y=(y_1,y_2,y_3)$ is:
 $$x\wedge y=\left|\begin{array}{ccc}
                                                                                        -e_1 & e_2 & e_3 \\
                                                                                        x_1 & x_2 & x_3 \\
                                                                                        y_1 & y_2 & y_3
                                                                                      \end{array}\right|,$$
where $\{ e_1, e_2, e_3\}$ is the standard basis of $\mathbb{R}^{3}$. For basic concepts and details of properties, see \cite{Ratcliffe}.

We have the following pseudo-spheres in $\mathbb{R}^{n+1}_1$ with centre $0$ and radius $r>0,$
 $$H^{n}(-r)=\{x\in\mathbb{R}^{n+1}_1\,\,|\,\,\langle x,x\rangle=-r^{2}\},\,\,\hbox{called \emph{Hyperbolic $n$-space}};$$
 $$S^{n}_1(r)=\{x\in\mathbb{R}^{n+1}_1\,\,|\,\,\langle x,x\rangle=r^2\},\,\,\hbox{called \emph{de Sitter $n$-space}};$$
 $$LC^{\ast}=\{x\in\mathbb{R}^{n+1}_1\setminus\{0\}\,\,|\,\,\langle x,x\rangle=0\},\,\,\hbox{called \emph{Lightcone}}.$$
 Instead of $S^{n}_1(1)$, we usually write $S^{n}_1$.

Let V be a vector subspace of $\mathbb{R}^{n+1}_1$. Then we say that V is  \emph{timelike} if and only if V has a timelike vector, \emph{spacelike} if and only if every non-zero vector in V is spacelike, or \emph{lightlike} otherwise. For a non-zero vector $v\in\mathbb{R}^{n+1}_1$ and a real number $c$, we define a \emph{hyperplane} with \emph{normal} $v$ by $$P(v,c)=\{x\in\mathbb{R}^{n+1}_1\,\,|\,\,\langle x,v\rangle=c\}.$$ We call $P(v,c)$ a spacelike hyperplane, a timelike hyperplane or lightlike hyperplane if $v$ is timelike, spacelike or lightlike, respectively.

We consider embeddings $\gamma: I\rightarrow \mathbb{R}^{n}_1$, where $I$ is an open interval of $\mathbb{R}$. The set $Emb(I,\mathbb{R}^{n}_1)$ of such embeddings is endowed with the Whitney $C^{\infty}$-topology. We say that a property is \emph{generic } if it is satisfied by curves in a residual subset of $Emb(I,\mathbb{R}^{n}_1)$.

We say that $\gamma$ is \emph{spacelike} (resp. \emph{timelike}) if $\gamma'(t)$ is a \emph{spacelike }(resp. \emph{timelike}) vector for all $t\in I$. A point $\gamma(t)$ is called a \emph{lightlike point} if $\gamma'(t)$ is a \emph{lightlike vector}.

As in \cite{Saloomtari} for plane curves, we define the subset $\Omega$ of $Emb(I, \mathbb{R}^{n}_1)$ such that a curve $\gamma$ is in $\Omega$ if and only if $\langle\gamma''(t),\gamma'(t)\rangle\neq 0$ whenever $\langle\gamma'(t),\gamma'(t)\rangle= 0$. One can show, using Thom's transversality results (see for example \cite{bruce}, Chapter 9), that $\Omega$ is a residual subset of $Emb(I,\mathbb{R}^{n}_1)$.
\begin{prop}\label{pro:isolados}
Let $\gamma\in\Omega$. Then the lightlike points of $\gamma$ are isolated points.
\end{prop}
\begin{proof}
The proof is similar to the case $n=2$ given in \cite{Saloomtari}.
\end{proof}
We observe that if the curve $\gamma\in\Omega$ then at a lightlike point $\gamma(t_0)$, the curve changes from a spacelike curve to a timelike curve or vice-versa.

To study the local properties of $\gamma$ at $\gamma(t_0)$, we
use the germ $\gamma:\mathbb{R},t_0\rightarrow \mathbb{R}^{3}_1$ of $\gamma$ at $t_0$.
The family of distance squared functions  $f:I\times\mathbb{R}^{3}_1\rightarrow\mathbb{R}$ on $\gamma$ is given by $$f(t,v)=\langle\gamma(t)-v,\gamma(t)-v\rangle.$$
We denote by  $f_v:I\rightarrow\mathbb{R}$ the function  $f_v(t)=f(t,v)$, for any fixed $v\in \mathbb{R}^{3}_1$.

The distance squared function $f_v$ has singularity of type $A_k$ at $t_0$ if the derivatives $(f_v)^{(p)}(t_0)=0$ for all $1\leq p\leq k$, and $(f_v)^{(k+1)}(t_0)\neq 0 $. We also say that $f_v$ has singularity of type $A_{\geq k}$ at $t_0$ if $(f_v)^{(p)}(t_0)=0$ for all $1\leq p\leq k$. This is valid including if $\gamma(t_0)$ is a lightlike point of the curve. Now let $F:\mathbb{R}^{3}_1\rightarrow \mathbb{R}$ be a submersion and $\gamma:I\rightarrow \mathbb{R}^{3}_1$ be a regular curve. We say that $\gamma$ and $F^{-1}(0)$ have contact of order $k$ or $k$-point contact at $t=t_0$ if the function  $g(t)=F\circ\gamma(t)$ satisfies $g(t_0)=g'(t_0)=\cdots=g^{(k)}(t_0)=0$ and $g^{(k+1)}(t_0)\neq0$, i.e, if $g$ has singularity $A_{k}$ at $t_0$. Then the singularity type of $f_v$ at $t_0$ measures the contact of $\gamma$ at $\gamma(t_0)$ with the pseudo-sphere of centre $v$ and radius $\parallel \gamma(t_0)-v\parallel$. The type of pseudo-sphere is determined by the sign of $\langle\gamma(t_0)-v,\gamma(t_0)-v\rangle$. For a generic curve in $\mathbb{R}^{3}_1$, $f_v$ has local singularities of type $A_1$, $A_2$, $A_3$ or $A_4$ (see \cite{peisano}), and the singularities $A_4$ occur at isolated points of the curve. If $f_{v_0}$ has an $A_k$-singularity ($k=2,3,4$) at $t_0$, then $f$ is a $(p)$-versal unfolding of $f_{v_0}$ \cite{peisano}.
The \emph{bifurcation set} of $f$ is given by

$$\mathfrak{Bif}(f)=\{v\in \mathbb{R}^{3}_1\,\,|\,\,f_v'(t)=f_v''(t)=0\,\,\mbox{in}\,\,(t,v)\,\,\mbox{for some}\,\,t\},$$ i.e., the directions where $f_v$ at $t$  has a degenerate (non-stable) singularity, that is, the singularity is of type  $A_{\geq2}$. It is defined even when the point is a lightlike point of $\gamma$.

The\emph{ focal set} of $\gamma$, for  spacelike or timelike curves, is the locus of centres of pseudo-spheres that has at least a $2$-point contact with the curve. Therefore, the $\mathfrak{Bif}(f)$ and the focal set of $\gamma$ coincide for  spacelike and timelike curves.

We have  a fundamental result of the unfolding theory:

\begin{theo}{\cite{bruce}} \label{teo:bruce} Let $G:(\mathbb{R}\times\mathbb{R}^3,(t_0,v_0))\rightarrow \mathbb{R}$ be a $3$-parameter unfolding of $g(t)$ which has an $A_k$-singularity at $t_0$. Suppose that $G$ is a $(p)$-versal unfolding, then $\mathfrak{Bif}(G)$ is locally diffeomorphic to
\begin{itemize}
  \item [(a)] $\mathbb{R}^2$, if $k=2$;
  \item [(b)] cuspidal edge $C\times\mathbb{R}$, if $k=3$;
  \item [(c)] swallowtail $SW$, if $k=4$,
\end{itemize}

\noindent where $C=\{(x_{1}, x_{2})\,|\,x_{1}^{2}=x_{2}^{3}\}$ is the ordinary cusp and $SW=\{(x_{1}, x_{2}, x_{3})\,|\,x_{1}=3u^{4}+u^{2}v,$ $x_{2}=4u^{3}+2uv,$ $x_{3}=v\}$ is the swallowtail (see \cite{bruce} for figure of the  SW surface).
\end{theo}


\section{The focal sets of spacelike and timelike curves}\label{sec:focalst}
Let $\gamma: I\rightarrow \mathbb{R}^{3}_1$ be a spacelike or a timelike curve and suppose that it is parametrised by arc length. This is possible because $\gamma$ has no lightlike points.

In this section, we remember  the Frenet-Serret formulae of  $\gamma$  and we find the parametrisation of their focal surfaces. Furthermore, we study the metric structure of these focal surfaces.

 We denote by $t$ the unit tangent vector to $\gamma$. Let $n$ be the unit normal vector to $\gamma$ given by $\gamma''(s)=k(s)n(s)$, where $k(s)=\|\gamma''(s)\|$ is defined as being the curvature of $\gamma$ at $s$, and $b(s)=t(s)\wedge n(s)$ the unit binormal vector to $\gamma(s)$. Then, we have the orthonormal basis $\{t(s),n(s), b(s)\}$ of $\mathbb{R}^3_1$ along $\gamma$. Using exactly  the same arguments as the case for a curve in an Euclidian 3-space, we have the following Frenet-Serret formulae (see \cite{izumyiatakiyama}, \cite{izumiya}):

 $$\left\{
   \begin{aligned}
     t'(s) & =k(s)\,n(s) \\
      n'(s) &=-\varepsilon(\gamma(s))\,\delta(\gamma(s))\,k(s)\,t(s) + \varepsilon(\gamma(s))\,\tau(s)\,b(s) \\
     b'(s) &=\tau(s)\,n(s)
   \end{aligned}
 \right.,$$

\noindent with $\tau(s)$ being the torsion of  $\gamma$ at $s$, $\varepsilon(\gamma(s))=sign(t(s))$, $\delta(\gamma(s))=sign(n(s))$, where $sign(v)$ is $1$ if the vector $v$ is spacelike or $-1$ if the vector $v$ is timelike. We call them, $\varepsilon$ and $\delta$ for short.

 Observe that if $\gamma$ is a spacelike or a timelike curve and $k(s)=0$ for some $s\in I$, then $f''_v(s)=\varepsilon(\gamma(s))\neq0$ and  there is no singularity $A_{\geq2}$.
Now if $\tau(s)=~0$ for some $s\in I$, then generically $f^{(3)}_v(s)=-\varepsilon(\gamma(s))k'(s)\neq0$, that is, there is no singularity $A_{\geq3}$. This is the reason for which $k(s)\neq0$ and $\tau(s)\neq0$ in the following proposition.

\begin{prop}{\cite{peisano}}\label{prop:peisano}
Let $\gamma:I\rightarrow\mathbb{R}^{3}_1$ be a spacelike or a timelike curve parametrised by arc length, with $k(s)\neq0$ and $\tau(s)\neq0$. Then
\begin{itemize}
  \item [(1)] $f_v'(s_0)=0$ if and only if there exist $\lambda$, $\mu$ $\in \mathbb{R}$ such that $\gamma(s_0)-v=\lambda n(s_0)+\mu b(s_0)$.
  \item [(2)] $f_v'(s_0)=f_v''(s_0)=0$  if and only if $v=\gamma(s_0)+\dfrac{\varepsilon(\gamma(s_0))}{\delta(\gamma(s_0))k(s_0)}n(s_0) + \mu b(s_0)$ for some $\mu \in \mathbb{R}$.
  \item [(3)] $f_v'(s_0)=f_v''(s_0)=f_v^{(3)}(s_0)=0$ if and only if

  $$v=\gamma(s_0)+\frac{\varepsilon(\gamma(s_0))}{\delta(\gamma(s_0))k(s_0)}n(s_0) + \frac{k'(s_0)}{\varepsilon(\gamma(s_0))\delta(\gamma(s_0))k^{2}(s_0)\tau(s_0)}b(s_0).$$
\end{itemize}
\end{prop}

Thus, for a spacelike or timelike curve $\gamma$ parametrised by arc length with $k(s)\neq0$, we have that the \emph{focal surface of $\gamma$} is given by

\begin{equation}\label{eq:bifr31}
\mathfrak{B}(s,\mu)=\gamma(s)+\frac{\varepsilon(\gamma(s))}{\delta(\gamma(s))k(s)}n(s) + \mu b(s),
\end{equation}

\noindent with $\mu \in\mathbb{R}$. The \emph{cuspidal curve} of the focal surface is given by
\begin{equation}\label{eq:cuspidalcurve}
\mathfrak{B}(s)=\gamma(s)+\frac{\varepsilon(\gamma(s))}{\delta(\gamma(s))k(s)}n(s) + \mu(s)b(s),
\end{equation}
 with $\mu(s)=\dfrac{k'(s)}{\varepsilon(\gamma(s))\delta(\gamma(s))k^{2}(s)\tau(s)}$, that is, where the distance squared function has singularity $A_{\geq3}$. We denote the cuspidal curve $\mathfrak{B}(s)$ by $\mathcal{C}$.

 We observe that the focal surface is a developable surface (for more details see \cite{izumiyadevelopable}).

\begin{prop}\label{prop:curvatimeesf} Let $\gamma$ be a connected timelike curve,
 then $\gamma$ does not intersect its focal surface.

\begin{proof}

 Suppose that $\gamma$ is timelike and intersects its focal surface, then there exists $s_1, s_2\in I$ with $s_1\neq s_2$ (for simplicity suppose that $s_2<s_1$) such that,
$$\gamma(s_1)-\frac{1}{k(s_1)}n(s_1)+\mu b(s_1)= \gamma(s_2).$$

Consider the function  $g:[s_2,s_1]\rightarrow \mathbb{R}$ given by  $g(s)=\langle \gamma(s), \gamma'(s_1)\rangle-\langle \gamma(s_1), \gamma'(s_1)\rangle$. Thus  $g(s_1)=g(s_2)=0$ and therefore by the Rolle's theorem exists $s_3\in(s_2,s_1)$ such that  $g'(s_3)=0$. Since $g'(s_3)=\langle \gamma'(s_1), \gamma'(s_3)\rangle$ we have that  $\gamma'(s_3)$ belongs to a plane generated by  $n(s_1)$ and $b(s_1)$. But  this is a
contradiction because this plane is spacelike and contains  $\gamma'(s_3)$ that is  timelike. Therefore, $\gamma$ does not intersect its focal surface.
\end{proof}
 \end{prop}

\begin{rem}\label{rem:curvatimeesf}
If $\gamma$ is spacelike and $\gamma$ intersects its focal surface, then generically its occurs at isolated points because  this is given by a generic transversality condition.
\end{rem}

 To study the  metric structure of the focal surface $\mathfrak{B}$, we need some concepts. A \emph{spacelike surface} is a surface for which the tangent plane, at any point, is a spacelike plane (i.e., consists only of spacelike vectors).
A \emph{timelike surface} is a surface for which the tangent plane, at any point, is a timelike plane (i.e., consists of  spacelike, timelike and lightlike vectors).

The pseudo scalar product in $\mathbb{R}^{3}_1$ induces a metric on the focal surface $\mathfrak{B}$ that may be degenerated at some points of $\mathfrak{B}$. This means that the tangent planes to $\mathfrak{B}$ are lightlike at these points.  We label the locus of such points the \emph{Locus of Degeneracy} and we denote  it by \emph{LD} (see \cite{farid} for Locus of Degeneracy of caustics of surfaces in $\mathbb{R}^3_1$). The LD of $\mathfrak{B}$ may be empty (Theorem \ref{curvatime}, \emph{(d)}) or a smooth curve (Proposition \ref{LD}) that splits the focal surface $\mathfrak{B}$ locally into a Riemannian (where the tangent planes are spacelike) and a Lorentzian region (where the tangent planes are timelike). It is interesting to study what happens at points where the metric is degenerate and explain the changes in the geometry from a Riemannian region to a Lorentzian region of the submanifold (see $\S\ref{sec:fslig}$). Furthermore, the focal surface can have points where the tangent plane is not defined.

Consider the focal surface of a spacelike or a timelike curve $\gamma$ , that is,

$$\mathfrak{B}(s,\mu)=\gamma(s)+\frac{\varepsilon(\gamma(s))}{\delta(\gamma(s))k(s)}n(s) + \mu b(s),\,\,\,\mu\in\mathbb{R}.$$

 Observe that $\mathfrak{B}_s=\dfrac{\partial \mathfrak{B}}{\partial s}(s,\mu)$ is parallel to $\mathfrak{B}_\mu=\dfrac{\partial \mathfrak{B}}{\partial \mu}(s,\mu)$ if and only if $$\mu(s)=\dfrac{k'(s)}{\varepsilon(\gamma(s))\delta(\gamma(s))k^2(s)\tau(s)},$$ and  $\mathfrak{B}(s,\mu(s))$ is the parametrisation of the curve where $f_v$ has  singularities of type $A_{\geq3}$, that is the cuspidal curve $\mathcal{C}$.

   Supposing  $$\mu(s)\neq\frac{k'(s)}{\varepsilon(\gamma(s))\delta(\gamma(s))k^2(s)\tau(s)},$$ then $\mathfrak{B}_s$ and $\mathfrak{B}_\mu$ generate the tangent planes of the surface $\mathfrak{B}$, and for $v=\lambda_1\mathfrak{B}_s +\lambda_2\mathfrak{B}_\mu$,
 \small{$$\langle v,v\rangle=\lambda_1^2\left(\dfrac{\tau^2}{k^2}\langle b,b\rangle+\dfrac{k'^2\delta}{k^4}-2\dfrac{\varepsilon k'\mu\tau}{k^2}+\mu^2\tau^2\delta\right)+2\lambda_1\lambda_2\left(\dfrac{\tau}{\delta k}\langle b,b\rangle\right)+\lambda_2^2\langle b,b\rangle.$$}
 We use this expression in the following theorem.

The above calculations show the item \emph{(a)} of the next result, that is, the tangent plane of the focal surface is not defined only at the points of the cuspidal curve $\mathcal{C}$ given by Equation \eqref{eq:cuspidalcurve}.




\begin{theo}\label{curvatime}
(a) Only at the points  of the cuspidal curve $\mathcal{C}$, the tangent planes of the focal surface are not defined.

Away from the cuspidal curve $\mathcal{C}$:
\begin{itemize}
 \item[(b)] the focal surface of a timelike generic curve  is spacelike;

 \item[(c)] the focal surface of a spacelike generic curve  is timelike;

 \item[(d)] if the curve is spacelike and timelike, then the LD set of the focal surface is empty.
\end{itemize}

\begin{proof}
(b) Let $\gamma$ be a timelike curve, then $n(s)$ and $b(s)$ are spacelike. Therefore,  $v=\lambda_1\mathfrak{B}_s +\lambda_2\mathfrak{B}_\mu$ are vectors on the tangent plane and
$$\langle v,v\rangle=\lambda_1^2\left(\left(\frac{k'}{k^2}+\mu\tau\right)^2+\frac{\tau^2}{k^2}\right)(s)+2\lambda_1\lambda_2\left(\frac{\tau}{k}\right)(s)+\lambda_2^2.\hspace{0.5cm}(\ast)$$

 Making $\langle v,v\rangle=0$, we can think in the above equation as a quadratic equation of $\lambda_1$, thus $\Delta=-4\lambda_2^2\left(\dfrac{k'}{k^2}+\mu\tau\right)^2(s)\leq0.$

Since $\mu(s)\neq-\dfrac{k'}{k^2\tau}(s)$ at the regular points of the focal surface, then $\Delta=0\Leftrightarrow\lambda_2=0$. Replacing $\lambda_2=0$ in $(\ast)$, we have a lightlike direction if $\tau(s)=0$ and $k'(s)=0$, but we are supposing $\tau(s)\neq0$ (see Proposition \ref{prop:peisano}), so $\Delta< 0$. Thus, we do not have lightlike directions in this plane, and therefore the tangent planes are spacelike.

(c)  Let $\gamma$ be a spacelike curve, then we have two cases: $n(s)$ timelike and $b(s)$ spacelike or $n(s)$ spacelike and $b(s)$ timelike.

In the case where $n(s)$ is timelike and $b(s)$ is spacelike, we have
$$\langle v,v\rangle=\lambda_1^2\left(\frac{\tau^2}{k^2}-\left(\frac{k'}{k^2}+\mu\tau\right)^2\right)(s)-2\lambda_1\lambda_2\left(\frac{\tau}{k}\right)(s)+\lambda_2^2.\hspace{0.5cm} (\ast\ast)$$
 Similar to (b), we obtain $ \Delta\geq0$ and at the regular points $\Delta=0\Leftrightarrow\lambda_2=0$. Replacing $\lambda_2=0$ in $(\ast\ast)$ we have a lightlike direction if $\left(\dfrac{\tau^2}{k^2}-\left(\dfrac{k'}{k^2}+\mu\tau\right)^2\right)(s)=0$, i.e.,   $\mu(s)=\mu_1(s)=\left(\dfrac{1}{k}-\dfrac{k'}{k^2\tau}\right)(s)$ or   $\mu(s)=\mu_2(s)=\left(-\dfrac{1}{k}-\dfrac{k'}{k^2\tau}\right)(s)$. Then

$$\mathfrak{B}_s(s,\mu_1(s))=\dfrac{\tau(s)}{k(s)}(n(s)-b(s))\,\,\mbox{and}\,\,\mathfrak{B}_s(s,\mu_2(s))=-\dfrac{\tau(s)}{k(s)}(n(s)+b(s))$$ are linearly independent lightlike vectors and $\mathfrak{B}_\mu(s,\mu_1(s))=\mathfrak{B}_\mu(s,\mu_2(s))=b(s)$. On the other hand, we have
$$\mathfrak{B}_s(s,\mu_2(s))=-\mathfrak{B}_s(s,\mu_1(s))-\dfrac{2\tau(s)}{k(s)}\mathfrak{B}_\mu(s,\mu_1(s)),$$ i.e., $\mathfrak{B}_s(s,\mu_2(s))$ belongs to the plane generated by $\mathfrak{B}_s(s,\mu_1(s))$ and $\mathfrak{B}_\mu(s,\mu_1(s))$. Similarly, the vector $\mathfrak{B}_s(s,\mu_1(s))$ is in the plane generated by $\mathfrak{B}_s(s,\mu_2(s))$ and $\mathfrak{B}_\mu(s,\mu_2(s))$.  Therefore if $\lambda_2=0$, the tangent planes at $\mathfrak{B}(s, \mu_1(s))$ and at  $\mathfrak{B}(s, \mu_2(s))$ will have two lightlike directions and then these planes will be timelike. Thus, we conclude that at all points of the focal surface, the tangent plane is timelike.

  When $n(s)$ is spacelike and $b(s)$ is timelike, it follows that at all points of the focal surface, the tangent plane is timelike, because $\mathfrak{B}_\mu(s,\mu)=b(s)$ where $b(s)$ is timelike.

(d) The LD set is empty and it is a consequence of (a), (b) and (c).

\end{proof}
\end{theo}
\section{The focal set near lightlike points}\label{sec:fslig}
  So far, we have studied what is happening to the focal surface of a spacelike or a timelike curve. The focal set is not defined at the lightlike points of the curve, but at these points, the bifurcation set of the family of distance squared functions on the curve is defined. Furthermore, the focal set is contained
   in the bifurcation set and they coincide if the curve is spacelike or timelike. Consider a curve with lightlike points. As the curve is in $\Omega$, these points are isolated (Proposition \ref{pro:isolados}) and the curve changes from spacelike to timelike at these points. We can think of the bifurcation set as a form of  pass from the focal set of the spacelike side of the curve to the focal set of the timelike side of the curve.
Our main goal in this section is then to understand this passage by  studying the geometry of the bifurcation set near the lightlike point of the curve. The principal result in this section is given by Theorem \ref{LD}.

 To study the bifurcation set near lightlike points $\gamma(t_0)$ of $\gamma$, we cannot parametrise the curve by the arc length since $\langle\gamma'(t_0),\gamma'(t_0)\rangle=0$. We consider then a smooth and regular curve $\gamma:I\rightarrow \mathbb{R}^3_1$ not parametrised by the arc length. The distance squared function is given by $f_v(t)=\langle\gamma(t)-v,\gamma(t)-v\rangle$. Thus $$\frac{1}{2}f_v'(t)=\langle\gamma(t)-v,\gamma'(t)\rangle.$$

It follows that $f_v$ is singular at $t$ if and only if $\langle\gamma(t)-v,\gamma'(t)\rangle=0$, equivalently,
$\gamma(t)-v=\mu N(t)+\lambda B(t)$, where $N(t)$ and $B(t)$ are vectors that generate the normal plane to the vector $\gamma'(t)$. (This condition includes the lightlike points.)

Differentiating again we obtain $$\frac{1}{2}f_v''(t)=\langle\gamma(t)-v,\gamma''(t)\rangle + \langle\gamma'(t), \gamma'(t)\rangle.$$

The singularity of $f_v$ is degenerate if and only if $f_v'(t)=f_v''(t)=0$, equivalently,
 $\gamma(t)-v=\mu N(t)+\lambda B(t)$ and
\begin{equation}\label{eq:bifvizlight}
\mu\langle N(t),\gamma''(t)\rangle+ \lambda\langle B(t), \gamma''(t)\rangle+ \langle\gamma'(t),\gamma'(t)\rangle=0.
\end{equation}

It follows that the bifurcation set of $f$ is given by
\begin{equation}\label{eq:biflightlike2}
\mathfrak{Bif}(f)=\{\gamma(t)-\mu N(t)-\lambda B(t)\,\,|\,\, (\mu,\lambda)\,\,\mbox{is\,\, solution\,\, of}\,\, \eqref{eq:bifvizlight}\}.
\end{equation}

 Away from the isolated lightlike points of $\gamma$, the bifurcation set is precisely the focal surface of the spacelike and timelike components of $\gamma$ studied in Section \ref{sec:focalst}.

Now our aim  is to study the  general expression \eqref{eq:biflightlike2} of the bifurcation set of the family of distance squared functions on the curve, to analyse what is happening with this surface when the curve $\gamma$ has lightlike points. We remember that since $\gamma \in \Omega$, then near a lightlike point, $\gamma$ changes from a spacelike curve to a timelike curve.

  Taking  $N(t)= \gamma'(t)\wedge\gamma''(t)$ and $B(t)=\gamma'(t)\wedge(\gamma'(t)\wedge\gamma''(t))$ and replacing  in \eqref{eq:bifvizlight} and \eqref{eq:biflightlike2}, we have that the bifurcation set of $f$ can be written as


$$\mathfrak{Bif}(f)=\{\gamma(t)-\mu N(t)-\frac{\langle\gamma'(t),\gamma'(t)\rangle}{\langle\gamma'(t)\wedge\gamma''(t),\gamma'(t)\wedge\gamma''(t)\rangle}B(t)\,\,|\,\, \mu\in\mathbb{R}\}.$$
In short, we will denote the map that defines $\mathfrak{Bif}(f)$ as $\mathfrak{B}(t,\mu)$, and the $\mathfrak{Bif}(f)$ set also as $\mathfrak{B}$.

 Since $\gamma\in\Omega$, at a lightlike point $\gamma(t_0)$ of $\gamma$, the vector $N(t_0)= \gamma'(t_0)\wedge\gamma''(t_0)$ is not lightlike, thus the bifurcation set above is well defined in the neighborhood of $t_0$. Furthermore, $B(t_0)$ is parallel to $\gamma'(t_0)$, and the vectors  $N(t_0)$ and $B(t_0)$ generate the normal plane to $\gamma'(t_0)$. In this case as $\langle\gamma'(t_0),\gamma'(t_0)\rangle=0$, then $\gamma'(t_0)$ is contained in this normal plane, which is a situation totally different from the Euclidian case.

Given a curve  $\gamma$ with lightlike points $\gamma(t_0)$, in the next result, we prove which types of singularities can  occur if  $v=\mathfrak{B}(t_0,\mu)$ for the distance squared function  $f_v$.  These are the only points of the bifurcation set that are not in the focal surfaces of the spacelike and timelike parts of the curve.

\begin{prop}\label{regular}
 Let $\gamma\in\Omega$. If $\gamma(t_0)$ is the lightlike point of  $\gamma$  and $v=\mathfrak{B}(t_0,\mu)$ then the distance squared function $f_v$  has  $A_2$-singularity except if  $\mu_0=\dfrac{-3\langle \gamma'(t_0),\gamma''(t_0)\rangle}{\langle\gamma'(t_0)\wedge\gamma''(t_0),\gamma'''(t_0)\rangle}$, where $f_{v_0}$ has  $A_{\geq3}$-singularity for $v_0=\mathfrak{B}(t_0,\mu_0)$.

\begin{proof}
 Consider $f_v(t)=\langle\gamma(t)-v,\gamma(t)-v\rangle$ the distance squared function on  $\gamma$. Then
$f_v^{(3)}(t)=6\langle\gamma'(t),\gamma''(t)\rangle+2\langle\gamma(t)-v, \gamma'''(t)\rangle$, i.e.,
$$f_v^{(3)}(t)=6\langle\gamma'(t),\gamma''(t)\rangle+2\langle\lambda(t)\gamma'(t)\wedge(\gamma'(t)\wedge\gamma''(t)) +\mu \gamma'(t)\wedge\gamma''(t),\gamma'''(t)\rangle.$$
Therefore, at lightlike point $\gamma(t_0)$, $f_v^{(3)}(t_0)=6\langle\gamma'(t_0),\gamma''(t_0)\rangle+2\langle\mu\gamma'(t_0)\wedge\gamma''(t_0),\gamma'''(t_0)\rangle.$ Remember that we are supposing as in Section \ref{sec:focalst} that the torsion is non zero at $t_0$ and therefore  $\langle\gamma'(t_0)\wedge\gamma''(t_0),\gamma'''(t_0)\rangle\neq0$. Then we have   $f_v^{(3)}(t_0)=0$ if and only if  $\mu=\dfrac{-3\langle \gamma'(t_0),\gamma''(t_0)\rangle}{\langle\gamma'(t_0)\wedge\gamma''(t_0),\gamma'''(t_0)\rangle}\neq0$, because $\gamma\in\Omega$.

\end{proof}
\end{prop}

 We analyse the curve $\mathfrak{B}(t_0,\mu)$, $\mu\in \mathbb{R}$, of the surface $\mathfrak{B}$,  that is a curve that split the focal surface of the spacelike side of  $\gamma$ of the focal surface of the timelike side of $\gamma$.

\begin{prop}\label{prop:curvaB}
(a) Let $\gamma:I\rightarrow \mathbb{R}^3_1$  be a regular curve with $\gamma(t_0)$ a lightlike point of $\gamma$. On the points of the curve  $\mathfrak{B}(t_0,\mu)$, $\mu\in \mathbb{R}$, surface $\mathfrak{B}$ has  degenerate tangent plane except for $\mu_0=\dfrac{-3\langle \gamma'(t_0),\gamma''(t_0)\rangle}{\langle\gamma'(t_0)\wedge\gamma''(t_0),\gamma'''(t_0)\rangle}$, where the tangent plane is not defined. Thus, the  $LD$  set of $\mathfrak{B}$ is  $\mathfrak{B}(t_0,\mu)$ with $\mu\neq\mu_0$.

\noindent (b) The curve $\mathfrak{B}(t_0,\mu)$, $\mu\in \mathbb{R}$, intersects the  cuspidal curve $\mathcal{C}$ when $\mu=\mu_0$.
\end{prop}

\begin{proof}
\emph{(a)} We have \begin{align*}
\frac{\partial\mathfrak{B}}{\partial t}(t,\mu)&=\gamma'(t)-\mu\left(\gamma'(t)\wedge\gamma'''(t)\right)-\lambda'(t)\left(\gamma'(t)\wedge(\gamma'(t)\wedge\gamma''(t))\right)-
\lambda(t)\left(\gamma'(t)\wedge(\gamma'(t)\wedge\gamma''(t))\right)' \\
\frac{\partial\mathfrak{B}}{\partial \mu}(t,\mu)&=-\left(\gamma'(t)\wedge\gamma''(t)\right),\,\,\,\,\hbox{where}
\end{align*}
$${\lambda'(t)}=\frac{2\langle\gamma'(t),\gamma''(t)\rangle\langle\gamma'(t)\wedge\gamma''(t),\gamma'(t)\wedge\gamma''(t)\rangle-2
\langle\gamma'(t),\gamma'(t)\rangle\langle\gamma'(t)\wedge\gamma''(t),\gamma'(t)\wedge\gamma'''(t)\rangle}{\left(\langle\gamma'(t)\wedge\gamma''(t),
\gamma'(t)\wedge\gamma''(t)\rangle\right)^2}$$

Then $\dfrac{\partial\mathfrak{B}}{\partial t}(t_0,\mu)=3\gamma'(t_0)-\mu\left(\gamma'(t_0)\wedge\gamma'''(t_0)\right)$ and $\dfrac{\partial\mathfrak{B}}{\partial \mu}(t_0,\mu)=-\left(\gamma'(t_0)\wedge\gamma''(t_0)\right),$ and therefore

$\dfrac{\partial\mathfrak{B}}{\partial t}(t_0,\mu)\wedge\dfrac{\partial\mathfrak{B}}{\partial \mu}(t_0,\mu)= (3\langle\gamma'(t_0),\gamma''(t_0)\rangle+\mu\langle\gamma'(t_0)\wedge\gamma''(t_0),\gamma'''(t_0)\rangle)\gamma'(t_0)$. As the torsion is non zero at $t_0$, we have that the vectors $\gamma'(t_0)$, $\gamma''(t_0)$ and $\gamma'''(t_0)$ are  linearly independent and then  $\dfrac{\partial\mathfrak{B}}{\partial t}(t_0,\mu)$ and $\dfrac{\partial\mathfrak{B}}{\partial \mu}(t_0,\mu)$ are linearly dependent if and only if $\mu=\dfrac{-3\langle \gamma'(t_0),\gamma''(t_0)\rangle}{\langle\gamma'(t_0)\wedge\gamma''(t_0),\gamma'''(t_0)\rangle}$, which we call $\mu_0$.

Furthermore, using Theorem \ref{curvatime}, we have that the focal surface of the spacelike side of the curve is timelike and that the focal surface  of the   timelike side of the curve is spacelike. Thus, the LD is contained in  $\mathfrak{B}(t_0,\mu)$.

Supposing  $\mu\neq\mu_0$,
the vectors of the tangent planes at the points of the curve $\mathfrak{B}(t_0,\mu)$ are given by: $$v=\lambda_1(3\gamma'(t_0)-\mu\left(\gamma'(t_0)\wedge\gamma'''(t_0)\right))-\lambda_2\left(\gamma'(t_0)\wedge\gamma''(t_0)\right).$$
 Then
\begin{align*}
\langle v,v\rangle &=\lambda_1^2\mu^2\langle\gamma'(t_0)\wedge\gamma'''(t_0),\gamma'(t_0)\wedge\gamma'''(t_0)\rangle  +2\lambda_1\lambda_2\mu\langle\gamma'(t_0)\wedge\gamma'''(t_0),\gamma'(t_0)\wedge\gamma''(t_0)\rangle +\\
& \lambda_2^2\langle\gamma'(t_0)\wedge\gamma''(t_0),\gamma'(t_0)\wedge\gamma''(t_0)\rangle.
\end{align*}
Making $\langle v,v\rangle=0$ and thinking of the above equation as a quadratic equation of  $\lambda_2$, we have $\Delta=0$.
Therefore, each  tangent plane  has a unique lightlike direction, given by $$(\lambda_1,\lambda_2)=\left(\lambda_1 ,-\lambda_1\dfrac{\mu\langle\gamma'(t_0)\wedge\gamma'''(t_0),\gamma'(t_0)\wedge\gamma''(t_0)\rangle}{\langle\gamma'(t_0)
\wedge\gamma''(t_0),\gamma'(t_0)\wedge\gamma''(t_0)\rangle }\right),$$ with $\lambda_1\neq0$. Thus, the induced metric on these planes is degenerate and the curve $\mathfrak{B}(t_0,\mu)$, with $\mu\neq \mu_0$,  is the $LD$ of the surface  $\mathfrak{B}$. (Observe that the denominator is different from zero because $\gamma\in\Omega$.)

\emph{(b)} The proof of this case  follows from  Proposition \ref{regular}, where we have that  $f_{v_0}$ has singularity $A_{\geq3}$ for $v_0=\mathfrak{B}(t_0,\mu_0)$, and the  cuspidal curve is precisely given by $ v's$ where $f_v$ has singularity $A_{\geq3}$.
\end{proof}

We  prove below that the surface   $\mathfrak{B}$ intersects the curve $\gamma$ at the  lightlike points and we study the geometric behavior of $\mathfrak{B}$ at the neighborhood of these points.

\begin{theo}\label{LD} Let $\gamma\in\Omega$, with  $\gamma(t_0)$ lightlike point, and  let  $\mathfrak{B}$ be the bifurcation set of the family of distance squared functions on $\gamma$. Then
\begin{itemize}
\item [(1)] the  surface $\mathfrak{B}$ intersects the curve $\gamma$ locally only at the lightlike point  $\gamma(t_0)$.

\item[(2)] the surface $\mathfrak{B}$ is regular at $\gamma(t_0)$.

\item [(3)] the  tangent line to the curve at $\gamma(t_0)$ is contained in the tangent plane to $\mathfrak{B}$ at such a point, that is, the unique lightlike direction of the tangent plane of $\mathfrak{B}$ at $\gamma(t_0)$ is the direction of the tangent line of $\gamma$ at $\gamma(t_0)$.

 \item[(4)] the $LD$ set of the  surface $\mathfrak{B}$ is a normal line to the curve passing through $\gamma(t_0)$ and splits the focal surface into a Riemannian and a Lorentzian region.

\end{itemize}
\begin{proof}
(1) Since $\gamma\in\Omega$, we have that
 at a lightlike point $\gamma(t_0)$,  $$\frac{\langle\gamma'(t_0),\gamma'(t_0)\rangle}{\langle\gamma'(t_0)\wedge\gamma''(t_0),\gamma'(t_0)\wedge\gamma''(t_0)\rangle}=0.$$ Then $\mathfrak{B}(t_0,0)=\gamma(t_0)$. Locally this intersection occurs only at the lightlike points because of Proposition \ref{prop:curvatimeesf} and Remark \ref{rem:curvatimeesf}.

(2) From Proposition \ref{regular}, we have that at $\gamma(t_0)=\mathfrak{B}(t_0,0)=v_0$, $f_{v_{0}}$  has only singularity of type $A_2$. Thus, by  Theorem \ref{teo:bruce}, surface $\mathfrak{B}$ is locally regular at this point.

 (3) Observe that $\gamma'(t_0)$ belongs to the tangent plane of the surface at $\gamma(t_0)$, which is generated by
$\dfrac{\partial\mathfrak{B}}{\partial t}(t_0,0)=3\gamma'(t_0)$ and $\dfrac{\partial\mathfrak{B}}{\partial \mu}(t_0,0)=-\left(\gamma'(t_0)\wedge\gamma''(t_0)\right)$.

Furthermore, the vectors of the tangent plane to the surface at $\gamma(t_0)$, are given by: $$v=3\lambda_1\gamma'(t_0)-\lambda_2\left(\gamma'(t_0)\wedge\gamma''(t_0)\right),$$ where $\lambda_1,\lambda_2\in\mathbb{R}$, and $\langle v,v\rangle=\lambda_2^2\langle \gamma'(t_0),\gamma''(t_0)\rangle^2\geq 0$. Thus $\gamma'(t_0)$ is the unique lightlike direction of the tangent plane, i.e. the tangent plane at $\gamma(t_0)$ is lightlike.

(4) As $\mathfrak{B}(t_0,\mu)=\gamma(t_0)-\mu N(t_0)$, then this normal line of $\gamma$ is contained in the focal surface.  By Proposition \ref{prop:curvaB} \emph{(a)},  the $LD$ is $\mathfrak{B}(t_0,\mu)$ except when $\mu_0=\dfrac{-3\langle \gamma'(t_0),\gamma''(t_0)\rangle}{\langle\gamma'(t_0)\wedge\gamma''(t_0),\gamma'''(t_0)\rangle}$, which is different from zero because $\gamma\in\Omega$. Therefore, near the $(t_0,0)$, i.e, near the $\mathfrak{B}(t_0,0)=\gamma(t_0)$, the induced metric along this normal line is degenerate. For this, it is enough to take a neighborhood of  $(t_0,0)$ that does not contain $\mu_0$.

\end{proof}
\end{theo}

%
\begin{figure}[H]
\center
\includegraphics[scale=0.4]{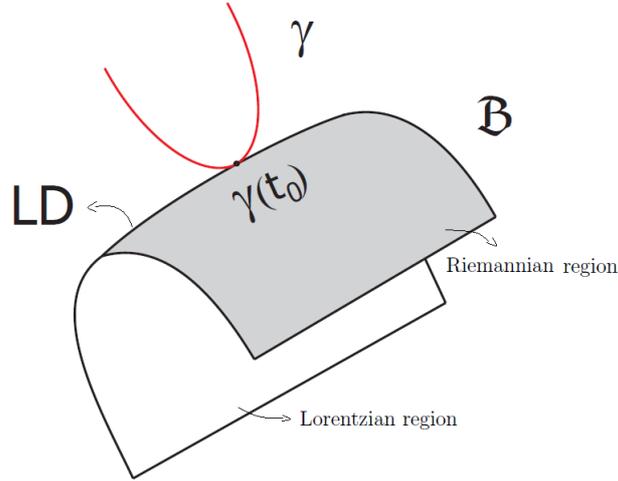}
\caption{Metric structure of the focal surface locally at a lightlike point of $\gamma$.}
\end{figure}

Observe that the  cuspidal curve $\mathcal{C}$ intersects the curve $\mathfrak{B}(t_0,\mu)$ at $\mathfrak{B}(t_0,\mu_0)$  (Proposition \ref{prop:curvaB} \emph{(b)}), i.e., away  from lightlike points where $\mu=0$.  In this case,   the   local  configuration of the bifurcation set at $\mathfrak{B}(t_0,\mu_0)=v_0$ is as in Figure \ref{ld2}, if $f_{v_0}$ has singularity  $A_3$ at $t_0$.
\begin{figure}[H]
\center
\includegraphics[scale=0.5]{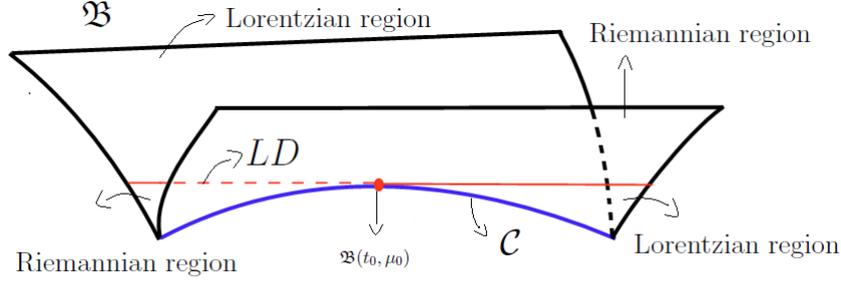}
\caption{ Example of a metric structure of a surface $\mathfrak{B}$ locally at no lightlike point $\mathfrak{B}(t_0,\mu_0)$.}
\label{ld2}
\end{figure}

In the case where   $f_{v_0}$ has  singularity   $A_4$ at $t_0$, for $\mathfrak{B}(t_0,\mu_0)=v_0$, the curve $LD$ intersects the  cuspidal curve at the singular point  of $\mathcal{C}$.

\section{Focal set of curves in $S^2_1$}\label{sec:focals21}

In this section, we consider curves in the de Sitter space $S^2_1\subset \mathbb{R}^3_1$ and their focal sets also in   $S^2_1$, which we call the spherical focal curve. To obtain  the results, we have Section \ref{sec:fslig} as a motivation.

Let $\gamma:I\rightarrow S^2_1$ be a spacelike or a timelike smooth and regular curve in $S^2_1$  parametrised by the arc length. For this curve, consider the orthonormal basis $\{\gamma(s),t(s)=\gamma'(s),n(s)=\gamma(s)\wedge t(s)\}$ of $\mathbb{R}^3_1$ along $\gamma$. By standard arguments, we have the following Frenet-Serret  type formulae:

 $$\left\{
   \begin{aligned}
     \gamma'(s) &=t(s) \\
      t'(s) &= -\varepsilon(\gamma(s))\,\gamma(s)+\delta(\gamma(s))\,k_g(s)\,n(s) \\
    n'(s) &=-\varepsilon(\gamma(s))\,k_g(s)\,t(s)
   \end{aligned}
 \right.,$$

\noindent where $\varepsilon(\gamma(s))=sign(t(s))$, $\delta(\gamma(s))=sign(n(s))$ and $k_g(s)=\langle\gamma''(s),n(s)\rangle$ is the geodesic curvature of $\gamma$ at $s$.

Consider the family of distance squared functions $f:I\times S^2_1\rightarrow\mathbb{R}$ on $\gamma$, given by $$f(s,v)=\langle\gamma(s)-v,\gamma(s)-v\rangle,$$

\noindent and $f_v:I\rightarrow \mathbb{R}$  given by $f_v(s)=f(s,v)$, for some $v\in S^2_1$ fixed.

 The \emph{spherical bifurcation set} of $f$ is given by $$\mathfrak{Bif}(f)=\{v\in S^{2}_1\,\,|\,\,f_v'(s)=f_v''(s)=0\,\, \mbox{in}\,\, (s,v)\,\,\mbox{for\,\, some}\,\, s\},$$ i.e, the directions where the singularity of $f$ at $s$ is at least $A_2$. The \emph{spherical focal curve of $\gamma$} is given by the spherical bifurcation set of $f$. Furthermore,  the spherical focal curve is the intersection of the focal surface in $\mathbb{R}^3_1$ with the  de Sitter space $S^2_1$.  Observe that since  $-\dfrac{1}{2}f_v(s)=\langle\gamma(s),v\rangle-1$ if $v\in S^2_1$, then  the singularities of the distance squared function and of the height function are the same. Therefore, the evolutes of a curve $\gamma$ in $S^2_1$, coincide with the spherical focal curve of $\gamma$.
 In \cite{Izumiyapeisanotorri}, the authors study the evolutes of hyperbolic plane curves, that is, a curve in $H^{2}(-1)$ and these evolutes  also coincide with the bifurcation set in $H^{2}(-1)$.

For a spacelike or a timelike curve  $\gamma$  parametrised by the arc length with $k_g(s)\neq 0$, we have that the spherical focal curve of $\gamma$ is given by $$\alpha^{\pm}(s)=\pm\frac{k_g(s)}{\sqrt{k_g^2(s)+\delta(\gamma(s))}}\gamma(s)\pm \frac{\varepsilon(\gamma(s))}{\sqrt{k_g^2(s)+\delta(\gamma(s))}}n(s).$$

\begin{rem}
To define the spherical focal curve, we must have  $k_g^2(s)+\delta(\gamma(s))>0$. Then in the case that $\gamma$ is spacelike, we must have $k_g(s)<-1$ or $k_g(s)>1$ and in the  case that  $\gamma$ is timelike, the spherical focal curve is always defined. Furthermore, as $\alpha^{-}(s)=-\alpha^{+}(s)$ we work only with $\alpha^{+}(s)$.
\end{rem}
Consider $\mathcal{C}$ the  cuspidal curve of the focal surface of $\gamma$ in $\mathbb{R}^3_1$, as in Section \ref{sec:focalst}. Then, we have the next result.
\begin{prop}
The singular points of the  spherical focal curve of $\gamma$ are given by $S^2_1\cap \mathcal{C}$.
\end{prop}
\begin{proof}
 Observe that $f_v$ has singularity $A_{\geq3}$ at $s_0$ if and only if $k_{g}'(s_0)=0$,  equivalently  $\alpha^{+}(s_0)$ (and $\alpha^{-}(s_0)$)  is the singular point of the spherical focal curve, because \vspace{0.5cm}

$(\alpha^{+})'(s)=\dfrac{\delta(\gamma(s))k_g'(s)}{(k_g^2(s)+\delta(\gamma(s)))\sqrt{k_g^2(s)+\delta(\gamma(s)))}}\gamma(s)-\dfrac{\varepsilon(\gamma(s))
k_g(s)k_g'(s)}{(k_g^2(s)+\delta(\gamma(s)))\sqrt{k_g^2(s)+\delta(\gamma(s))}}n(s).$
\end{proof}
In the next proposition, we study the metric structure of the spherical focal curve of a spacelike curve and of a timelike curve.

\begin{prop} Away from the singular points,
\begin{itemize}
\item[(a)] the spherical focal curve of a spacelike curve is  timelike;

\item[(b)] the spherical focal curve of a timelike curve is spacelike.
\end{itemize}
\begin{proof}

%
%
(a)  Away from the singular points of $\alpha^{+}$, i.e., where $k_g'(s)\neq0$ we have that $\alpha^{+}$ is a timelike curve because
$$\langle(\alpha^{+})'(s),(\alpha^{+})'(s)\rangle=\frac{-(k_g')^2(s)}{(k_g^2(s)-1)^2}<0.$$

(b) This proof is analogous to the case (a).

\end{proof}
\end{prop}

We want to know what is happening at the lightlike points of  a curve $\gamma$. For this, let us find an expression of the bifurcation set near a lightlike point of $\gamma$. Here we cannot consider $\gamma:I\rightarrow S^2_1$ parametrised by the arc length and the focal set is not defined at the lightlike point.

  Let $ \gamma(t)$ and $N(t)=\gamma(t)\wedge \gamma'(t)$ be the vectors that generate the normal plane to the vector $\gamma'(t)$ and consider the family of distance squared functions $f:I\times S^2_1\rightarrow \mathbb{R}$ on $\gamma$. By definition, we have that the bifurcation set of $f$ is given by
$$\mathfrak{Bif}(f)=\{\pm\sqrt{1+\mu^2\langle\gamma'(t),\gamma'(t)\rangle}\gamma(t)+\mu N(t)\,|\, \mu\,\,\mbox{is\,\,the\,\, solution\,\, of\,\, the\,\, equation}\,\, (1^{\pm})\},$$ where
$\mu\langle\gamma(t)\wedge\gamma'(t),\gamma''(t)\rangle\pm\sqrt{1+\mu^2\langle\gamma'(t),\gamma'(t)\rangle}\langle\gamma(t),\gamma''(t)\rangle=0.\,\,\,\,\,\,\,\,(1^{\pm})$\\


 \begin{rem}\label{obs:denominador}
  Let $\gamma\in \Omega$ such that  $\gamma(t_0)$ is a  lightlike point of $\gamma$. In the next result, we use that $\langle\gamma(t_0)\wedge\gamma'(t_0),\gamma''(t_0)\rangle\neq0$.
Indeed,  if  $\langle\gamma(t_0)\wedge\gamma'(t_0),\gamma''(t_0)\rangle=0$ then exist  $a$, $b\in \mathbb{R}$ with $a^2+b^2\neq0$ such that  $\gamma''(t_0)=a\gamma(t_0)+b\gamma'(t_0)$, because $\gamma(t_0)$ and $\gamma'(t_0)$ are vectors  linearly independent.
By supposing $a\neq0$ and  $b=0$, then $\langle\gamma''(t_0),\gamma(t_0)\rangle=a\neq0$, that is a contradiction, since $\langle\gamma(t),\gamma''(t)\rangle=-\langle\gamma'(t),\gamma'(t)\rangle$.
Now suppose that $a=0$  and $b\neq0$, then $\langle\gamma''(t_0),\gamma'(t_0)\rangle=b\langle\gamma'(t_0),\gamma'(t_0)\rangle=0$, that is a contradiction, because $\gamma\in\Omega$.  For $a\neq0$ and $b\neq0$, we  get the same contradictions. Therefore, $\langle\gamma(t_0)\wedge\gamma'(t_0),\gamma''(t_0)\rangle\neq0$.
\end{rem}

Solving the equation $(1^{+})$ and using the fact that $\langle\gamma(t),\gamma''(t)\rangle=-\langle\gamma'(t),\gamma'(t)\rangle$, it follows that the solutions are $\mu(t)$ or $-\mu(t)$ where $$\mu(t)=\dfrac{\langle\gamma'(t),\gamma'(t)\rangle}{\sqrt{\langle\gamma(t)\wedge\gamma'(t),\gamma''(t)\rangle^2-\langle\gamma'(t),\gamma'(t)\rangle^3}}.$$

Observe that in the neighborhood of $t_0$, the term inside of the root  of the denominator  is greater than zero, because of Remark \ref{obs:denominador}  $\langle\gamma(t_0)\wedge\gamma'(t_0),\gamma''(t_0)\rangle\neq0$.

 If $\langle\gamma(t_0)\wedge\gamma'(t_0),\gamma''(t_0)\rangle>0$ then $\mu(t)$ is the solution of  $(1^{+})$ and $-\mu(t)$ is the solution of  $(1^{-})$. If $\langle\gamma(t_0)\wedge\gamma'(t_0),\gamma''(t_0)\rangle<0$ then $-\mu(t)$ is the solution of $(1^{+})$ and $\mu(t)$ is the solution of  $(1^{-})$.
Therefore, we have that  $\alpha^{+}(t)$ is a smooth curve and we can  rewrite
 $\alpha^{+}(t)$ as
 $$\sqrt{1+\mu^2(t)\langle\gamma'(t),\gamma'(t)\rangle}\gamma(t)+\mu(t) N(t)\,\,\hbox{if}\,\,\langle\gamma(t_0)\wedge\gamma'(t_0),\gamma''(t_0)\rangle>0\,\,\,\,\,\hbox{or}$$
$$\sqrt{1+\mu^2(t)\langle\gamma'(t),\gamma'(t)\rangle}\gamma(t)-\mu(t) N(t)\,\,\hbox{if}\,\, \langle\gamma(t_0)\wedge\gamma'(t_0),\gamma''(t_0)\rangle<0.$$

 The  above bifurcation set is contained in  $S^2_1$ and then we have a  spherical curve of $\gamma$ given by  $\mathfrak{Bif}(f)=\alpha^{+}\cup\alpha^{-}$, where $\alpha^{+}$ and $\alpha^{-}$ are symmetric.

\begin{prop}\label{proesfe}
The spherical  curve  $\alpha^{+}$ is a smooth curve that  intersects the curve $\gamma$ at the lightlike points of $\gamma$. The curve $\alpha^{-}$ does not intersect the curve $\gamma$, but it has the same geometry of $\alpha^{+}$, by symmetry.
\end{prop}

\begin{proof}
Let $\gamma(t_0)$ be a lightlike point of $\gamma$. The parametrisation of the spherical  curve $\alpha^{+}$  locally at $t_0$, is
given as above and a proof of the proposition follows directly from the substitution $t=t_0$ at  $\alpha^{+}$.

\end{proof}

 Here, we have an example of a spherical curve $\alpha^{+}$ of the curve $\gamma(t)=(t^2-t,t^2+t,\sqrt{1-4t^3})$ in $S^3_1$.  We use the Maple software  to obtain  the complicated expression of $\alpha^{+}$ and of the surface  $\mathfrak{B}$ (that we omit here) and the Figure below. $\mathfrak{B}$ is the bifurcation set given in Section \ref{sec:fslig}.

\begin{figure}[H]
\center
\includegraphics[scale=0.7]{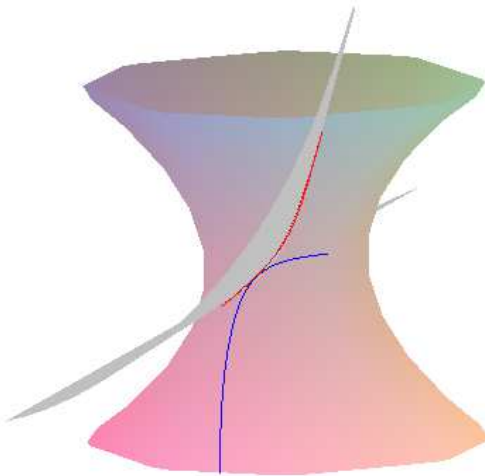}
\caption{Example of a spherical curve $\alpha^{+}$ (curve $\gamma$ is  blue, $\alpha^{+}$ is red and  $\mathfrak{B}$ is the gray surface).}
\end{figure}

\begin{rem}
The curve $\alpha^{-}$ does not intersect the curve $\gamma$, but $\alpha^{-}$ intersects $-\gamma$ at the lightlike point $-\gamma(t_0)$. The focal set of $\gamma$ and of $-\gamma$ are the same.
\end{rem}

\section{Focal set of curves in $S^3_1$}\label{sec:focals31}
In this section, we  consider curves in \emph{de Sitter space} $S^3_1\subset \mathbb{R}^4_1$ and we study the focal set in $S^3_1$ of these curves.  To obtain  the results for curves in the de Sitter space, we have Section \ref{sec:focals21} as the motivation.
Let $\gamma:I\rightarrow S^3_1$ be a  smooth and regular curve in  $S^3_1$. In the case where the  curve is spacelike or timelike, we can parametrise it by the arc length s. Thus, for the spacelike curve, we take the unit tangent vector  $t(s)=\gamma'(s)$.  Suppose that $\langle t'(s),t'(s)\rangle\neq 1$, then $\parallel t'(s)+\gamma(s)\parallel\neq0$, and we have other unit vector  $n(s)=\dfrac{t'(s)+\gamma(s)}{\parallel t'(s)+\gamma(s)\parallel}$. We also define  a unit vector by $e(s)=\gamma(s)\wedge t(s)\wedge n(s)$, then we have an orthonormal basis  $\{\gamma(s),t(s),n(s),e(s)\}$ of $\mathbb{R}^4_1$ along  $\gamma$. The Frenet-Serret  type formulae of a spacelike curve in $S^3_1$ (see \cite{Fusho}), are given by

$$\left\{\begin{aligned}
     \gamma'(s) &=t(s) \\
      t'(s) &= -\gamma(s)+k_g(s)\,n(s) \\
    n'(s) &=-\delta(\gamma(s))\,k_g(s)\,t(s)+\tau_g(s)\,e(s)\\
  e'(s) &=\tau_g(s)n(s)
\end{aligned}\right.,$$

\noindent where $\delta(\gamma(s))=sign(n(s))$, $k_g(s)=\parallel t'(s)+\gamma(s)\parallel$, $\tau_g(s)=\dfrac{\delta(\gamma(s))}{k_g^2(s)}\det(\gamma(s),\gamma'(s),\gamma''(s),$ $\gamma'''(s))$, and $\det$ is the determinant of the $4\times4$ matrix. Here $k_g$ is called the geodesic curvature and $\tau_g$  the geodesic torsion of $\gamma$ (see \cite{Fusho}).

Since $\langle t'(s)+\gamma(s),t'(s)+\gamma(s)\rangle=\langle t'(s),t'(s)\rangle-1$, the condition $\langle t'(s),t'(s)\rangle\neq 1$ is equivalent to the condition $k_g(s)\neq0$.

If the curve is timelike, we take the unit tangent vector $t(s)=\gamma'(s)$. By supposing a generic condition $\langle t'(s),t'(s)\rangle\neq 1$, then $\parallel t'(s)-\gamma(s)\parallel\neq0$, and we have other unit vector  $n(s)=\dfrac{t'(s)-\gamma(s)}{\parallel t'(s)-\gamma(s)\parallel}$. We also define an unit vector by $e(s)=\gamma(s)\wedge t(s)\wedge n(s)$, then we have an orthonormal basis $\{\gamma(s),t(s),n(s),e(s)\}$ of $\mathbb{R}^4_1$ along $\gamma$. Thus, the Frenet-Serret type formulae of a timelike curve in $S^3_1$ are given by

  $$\left\{
   \begin{aligned}
     \gamma'(s) &=t(s) \\
      t'(s) &= \gamma(s)+k_h(s)\,n(s)\\
    n'(s) &= k_h(s)\,t(s)+\tau_h(s)\,e(s)\\
  e'(s) &=-\tau_h(s)\,n(s)
   \end{aligned}
 \right.,$$

\noindent where $k_h(s)=\parallel t'(s)-\gamma(s)\parallel$ and $\tau_g(s)=-\dfrac{1}{k_h^2(s)}\det(\gamma(s),\gamma'(s),\gamma''(s),\gamma'''(s))$. Here $k_h$ is called the hyperbolic curvature and $\tau_h$ the  hyperbolic torsion of $\gamma$ (see \cite{kh}).

Since $\langle t'(s)-\gamma(s),t'(s)-\gamma(s)\rangle=\langle t'(s),t'(s)\rangle-1$, the condition $\langle t'(s),t'(s)\rangle\neq 1$ is equivalent to the condition $k_h(s)\neq0$.

Consider the family of distance squared functions, $f:I\times S^3_1\rightarrow\mathbb{R}$, on $\gamma$ $$f(s,v)=\langle\gamma(s)-v,\gamma(s)-v\rangle,$$

\noindent where $f_v(s)=f(s,v)$, for some $v\in S^3_1$ fixed. Observe that since $v\in S^3_1$, then $-\dfrac{1}{2}f_v(s)=\langle\gamma(s),v\rangle-1$ and the singularities of the distance squared function and the height function are the same.

 The \emph{spherical bifurcation set} of $f$ is given by $$\mathfrak{Bif}(f)=\{v\in S^{3}_1\,\,|\,\,f_v'(s)=f_v''(s)=0\,\, \mbox{at}\,\, (s,v)\,\,\mbox{for\,\,some}\,\, s\},$$ i.e., the directions where the singularity of $f$ at $s$ is $A_{\geq2}$. This is also defined for the lightlike points of $\gamma$.

 The spherical focal surface of $\gamma$ coincides with the spherical bifurcation set of $f$. Furthermore, for curves in $S^3_1\subset\mathbb{R}^4_1$ the spherical focal surface is the intersection of the focal hypersurface in $\mathbb{R}^4_1$ with the de Sitter space $S^3_1$.

For a spacelike curve $\gamma$  parametrised by the arc length with $k_g(s)\neq 0$, we have that the spherical focal surface of  $\gamma$ is given by

 $$\mathfrak{B}^{\pm}(s,\mu)=\mu\gamma(s)+\frac{\mu}{\delta(\gamma(s))k_g(s)}n(s)\pm\frac{\sqrt{-\delta(\gamma(s))k_g^2(s)+
\delta(\gamma(s))\mu^2(k_g^2(s)+\delta(\gamma(s))}}{k_g(s)}e(s),$$ with $\mu\in\mathbb{R}$. The \emph{g-spherical cuspidal curve} is given by $\mathfrak{B}^{\pm}(s,\mu(s))=\mathfrak{B}^{\pm}(s)$, where $$\mu(s)=\frac{\pm\tau_g(s)k_g^2(s)}{\sqrt{\tau_g^2(s)k_g^4(s)-k_g'^2(s)\delta(\gamma(s))-\tau_g^2(s)k_g^2(s)\delta(\gamma(s))}}.$$

\noindent For a timelike curve $\gamma$ parametrised by the arc length with $k_h(s)\neq 0$, the spherical focal surface is given by
$$\mathfrak{B}^{\pm}(s,\mu)=\mu\gamma(s)-\frac{\mu}{k_h(s)}n(s)\pm\frac{\sqrt{k_h^2(s)-\mu^2(k_h^2(s)+1)}}{k_h(s)}e(s),$$ with $\mu\in\mathbb{R}$. The \emph{h-spherical cuspidal curve} is given by $\mathfrak{B}^{\pm}(s,\mu(s))=\mathfrak{B}^{\pm}(s)$, where $$\mu(s)=\frac{\pm\tau_h(s)k_h^2(s)}{\sqrt{\tau_h^2(s)k_h^4(s)+k_h'^2(s)+\tau_h^2(s)k_h^2(s)}}.$$

\begin{rem}
The spherical focal surface of a spacelike curve is defined if $-\delta(\gamma(s))k_g^2(s)+
\delta(\gamma(s))\mu^2(k_g^2(s)+\delta(\gamma(s))\geq0$. As $k_g(s)\neq0$, we have that $n(s)$ is spacelike or timelike. In the case where $n(s)$ is spacelike, the spherical focal surface is defined for $$\mu\leq-\frac{k_g(s)}{\sqrt{k_g^2(s)+1}}\,\,\,or\,\,\,\mu\geq\frac{k_g(s)}{\sqrt{k_g^2(s)+1}},$$

\noindent otherwise in the case that $n(s)$ is timelike, the spherical focal surface is  defined for $$-\frac{k_g(s)}{\sqrt{k_g^2(s)-1}}\leq\mu\leq\frac{k_g(s)}{\sqrt{k_g^2(s)-1}}.$$
  The spherical focal surface of a timelike curve is defined if $$-\frac{k_h(s)}{\sqrt{k_h^2(s)+1}}\leq\mu\leq\frac{k_h(s)}{\sqrt{k_h^2(s)+1}}.$$ Furthermore, in both cases where $\mathfrak{B}^{+}(s,\mu)$ and $\mathfrak{B}^{-}(s,\mu)$ are symmetric, then we study only $\mathfrak{B}^{+}(s,\mu)$.

\end{rem}

We prove in the next results that the tangent plane of the spherical focal surface of a spacelike curve (respectively, timelike) is not defined at the points of the $g$-spherical cuspidal curve (respectively, of the $h$-spherical cuspidal curve). Furthermore, away from these curves we analyse  the metric structure of the spherical focal surface in each case.

 Away from the $g$-spherical cuspidal curve,  $$\mu(s)\neq\frac{\pm\tau_g(s)k_g^2(s)}{\sqrt{\tau_g^2(s)k_g^4(s)-k_g'^2(s)\delta(\gamma(s))-\tau_g^2(s)k_g^2(s)\delta(\gamma(s))}},$$ or away from the  $h$-spherical cuspidal curve,  $\mu(s)\neq\dfrac{\pm\tau_h(s)k_h^2(s)}{\sqrt{\tau_h^2(s)k_h^4(s)+k_h'^2(s)+\tau_h^2(s)k_h^2(s)}}$,
$v=\lambda_1\mathfrak{B}^{+}_s+\lambda_2\mathfrak{B}^{+}_{\mu}$, with $\lambda_1, \lambda_2\in\mathbb{R}$, are the vectors of the tangent plane of the spherical focal surface at $\mathfrak{B}^{+}(s,\mu)$ and $\langle v,v\rangle=\lambda_1^2\langle\mathfrak{B}^{+}_s,
\mathfrak{B}^{+}_s\rangle+2\lambda_1\lambda_2\langle\mathfrak{B}^{+}_s,\mathfrak{B}^{+}_{\mu}\rangle+\lambda_2^2\langle\mathfrak{B}^{+}_{\mu},
\mathfrak{B}^{+}_{\mu}\rangle$, by using the respective parametrisation of $\mathfrak{B}(s,\mu)$ for spacelike or timelike $\gamma$.

\begin{prop} \label{prop:gspherical} Let $\gamma$ be a spacelike curve.
\begin{itemize}
\item[(a)] The tangent plane of the spherical focal surface of $\gamma$ is not defined on the  $g$-spherical cuspidal curve.

\item[(b)] Away from the g-spherical cuspidal curve, the spherical focal surface of $\gamma$ is timelike.
\end{itemize}
\begin{proof}
(a) Considering a spacelike curve, the tangent plane at the points of the spherical focal surface is generated by the vectors

\begin{align*}
\mathfrak{B}^{+}_s(s,\mu)&= \left(\frac{-\mu k_g'(s)+\delta(\gamma(s)) \tau_g(s) k_g(s)\sqrt{-\delta(\gamma(s)) k_g^2(s)+
\delta(\gamma(s))\mu^2(k_g^2(s)+\delta(\gamma(s))})}{\delta k_g^2(s)}\right)n(s)\\&+\left(\frac{\mu\tau_g(s) k_g(s)\sqrt{-\delta(\gamma(s)) k_g^2(s)+
\delta(\gamma(s))\mu^2(k_g^2(s)+\delta(\gamma(s))})-\delta(\gamma(s)\mu^2k_g'(s))}{\delta(\gamma(s)) k_g^2(s)\sqrt{-\delta(\gamma(s)) k_g^2(s)+
\delta(\gamma(s))\mu^2(k_g^2(s)+\delta(\gamma(s))})}\right)e(s)\hspace{0.8cm}
 \end{align*}
and
$$\mathfrak{B}^{+}_\mu(s,\mu)=\gamma(s)+ \frac{1}{\delta(\gamma(s))k_g(s)}n(s)+\left(\frac{\delta(\gamma(s))\mu(k_g^2(s)+\delta(\gamma(s)))}{k_g(s)\sqrt{-\delta(\gamma(s))k_g^2(s)
+\delta(\gamma(s))\mu^2(k_g^2(s)+\delta(\gamma(s))}}\right)e(s).$$

The vectors $\mathfrak{B}^{+}_s$ and $\mathfrak{B}^{+}_\mu$ are linearly dependent if and only if $$\mu(s)=\frac{\pm\tau_g(s)k_g^2(s)}{\sqrt{\tau_g^2(s)k_g^4(s)-k_g'^2(s)\delta(\gamma(s))-\tau_g^2(s)k_g^2(s)\delta(\gamma(s))}}$$ that is precisely where the tangent plane is not  defined and furthermore is where $f_v$ has singularities of type  $A_{\geq3}$, that is the  $g$-spherical cuspidal curve.

(b)
Let $\gamma$ be a spacelike curve. Let us suppose that $n(s)$ is spacelike and $e(s)$ is timelike, thus
making $\langle v,v\rangle=0$, and thinking of this equation as a quadratic equation, then
$$\Delta=4\lambda_2^2\frac{(\tau_gk_g\sqrt{-k_g^2+\mu^2(k_g^2+1)}-\mu k_g')^2}{k_g^2(-k_g^2+\mu^2(k_g^2+1))}.$$

The tangent plane generated by $\mathfrak{B}^{+}_s$ and $\mathfrak{B}^{+}_\mu$ can be lightlike, if $\Delta=0$.
As we are supposing $\tau_g(s)k_g(s)\sqrt{-k_g^2(s)+\mu^2(k_g^2(s)+1)}-\mu k_g'(s)\neq0$, for  $\mathfrak{B}^{+}_s$ and $\mathfrak{B}^{+}_\mu$ be linearly independent, then we have $\Delta=0$ if and only if $\lambda_2=0$, that is, if $\mathfrak{B}^{+}_s$ is lightlike. But,  $\mathfrak{B}^{+}_s(s,\pm1)$ are the only lightlike vectors. Besides $\mathfrak{B}^{+}_{\mu}(s,\pm1)$ are timelike vectors, i.e., the tangent planes  at the points $(s,\pm1)$ are timelike.  Since, we have  $\Delta>0$ at the others points, thus the spherical focal surface is timelike.
\end{proof}
\end{prop}

\begin{prop}\label{prop:hspherical} Let $\gamma$ be a timelike curve.
\begin{itemize}
\item[(a)]  The tangent plane of the spherical focal surface of $\gamma$ is not defined on the $h$-spherical cuspidal curve.

\item[(b)] Away from the h-spherical cuspidal curve, the spherical focal surface of $\gamma$ is spacelike.
\end{itemize}

\begin{proof}
 The proofs are analogous to the proofs of  Proposition \ref{prop:gspherical}. In case (b), we observe that $\mathfrak{B}_s^{+}(s,\pm1)$ are not defined, then we prove that $\Delta<0$ for the equation $\langle v,v\rangle=0$.
%
%
%

\end{proof}
\end{prop}

Now our aim  is to find a general expression for the $\mathfrak{Bif}(f)$ to know what is happening to the spherical focal surface near a  lightlike point of $\gamma$. For this, consider the curve  $\gamma$  not parametrised by the arc length and  a vector $N(t)$ such that  $\gamma(t)$, $N(t)$, and $ E(t)=\gamma(t)\wedge \gamma'(t)\wedge N(t)$ generate the normal hyperplane to the vector $\gamma'(t)$. By definition, we have that $\mathfrak{Bif}(f)$ of $\gamma$ is a spherical  surface of $\gamma$ given by
$$\mathfrak{B}^{\pm}(t,\mu)=\mu\gamma(t)+\beta(t,\mu) N(t)+\lambda(t,\mu)E(t),$$ where  $\beta$ and $\lambda$ satisfies the  equations below:

$$\lambda(t,\mu)=\frac{\mu\langle\gamma'(t),\gamma'(t)\rangle-\beta\langle\gamma''(t), N(t)\rangle}{\langle\gamma''(t),E(t)\rangle}\hspace{0.8cm} and$$
 \noindent $\beta(t,\mu)$ is equal to

$\left(\frac{\mu\langle\gamma',\gamma'\rangle\langle\gamma'', N\rangle\langle E, E\rangle\pm\sqrt{\langle\gamma'', N\rangle^2\langle E, E\rangle\langle \gamma'',E\rangle^2(1-\mu^2)+\langle N, N\rangle\langle\gamma'', E\rangle^4(1-\mu^2)-\langle N,N\rangle\langle \gamma'',E\rangle^2\langle\gamma',\gamma'\rangle^2\langle E,E\rangle\mu^2}}{\langle\gamma'', E\rangle^2\langle N, N\rangle^2+\langle E, E\rangle\langle \gamma'', N\rangle^2}\right)(t)$

\noindent where $\mu$ is real number such that the root of $\beta$ is defined.
\begin{rem}
The spherical  surface $\mathfrak{B}^{\pm}$  is well defined near a lightlike point $\gamma(t_0)$. Let
 $R(t,\mu)=A(t)\mu^2 +B(t)$, the term inside the squared root of the above $\beta$, where
  \begin{align*}
  A(t)&=(-\langle\gamma'', N\rangle^2\langle E, E\rangle\langle \gamma'',E\rangle^2-\langle N, N\rangle\langle \gamma'',E\rangle^4-\langle N,N\rangle\langle \gamma'',E\rangle^2\langle\gamma',\gamma'\rangle^2\langle E,E\rangle)(t)\\
  \mbox{and}\hspace{0.4cm}B(t)&=(\langle\gamma'',E\rangle^2\langle  \gamma'',N\rangle^2 \langle E,E\rangle+\langle N, N\rangle\langle \gamma'',E\rangle^4)(t).
  \end{align*}
   Then for the spherical  surface $\mathfrak{B}^{\pm}$  be defined, we must have $R(t,\mu)\geq0$. Making the calculations at the lightlike point we have $A(t_0)<0$ and $B(t_0)>0$(these are equal in module) and in this case the spherical  surface $\mathfrak{B}^{\pm}$ is defined when $R(t_0,\mu)=A(t_0)\mu^2 +B(t_0)\geq0$, i.e., $-1\leq\mu\leq1$. Thus, for  continuity, there is a neighborhood near the $(t_0,\mu)$ such that  $R(t,\mu)\geq0$.

 \end{rem}

\begin{prop}
The spherical surface $\mathfrak{B}^{\pm}$  intersects the curve $\gamma$ at  lightlike points and the tangent planes to $\mathfrak{B}^{\pm}$  are not defined at these points.
\end{prop}

\begin{proof}

Let $\gamma(t_0)$ be a lightlike point of $\gamma$.  Analysing the expression of the spherical  surface  $\mathfrak{B}^{\pm}$ we have   $\mathfrak{B}^{\pm}(t_0,1)=\gamma(t_0)$, because  $\beta(t_0,1)=0$ and  $\lambda(t_0,1)=0$. Since   $R(t_0,1)=0$, then the  tangent planes to the spherical  surface at $\mathfrak{B}^{\pm}(t_0,1)$ are not defined.
\end{proof}

We observe that $\mathfrak{B}^{+}(t_0,-1)=\mathfrak{B}^{-}(t_0,-1)=-\gamma(t_0)$ and the bifurcation set of $\gamma$ and $-\gamma$ are the same. Furthermore $R(t_0,-1)=0$, then the tangent planes to the spherical  surface $\mathfrak{B}^{\pm}$  also are not defined at these points. Then, we have the next result.

\begin{prop}
The $LD$ set of the spherical  surface  $\mathfrak{B}^{\pm}$ are the curves $\mathfrak{B}^{\pm}(t_0,\mu)$, $-1<\mu<1$.
\end{prop}
\begin{proof}
 The  tangent planes at $\mathfrak{B}^{\pm}(t_0,\mu)$ exist for $-1<\mu<1$. The proof follows observing that the spherical surface $\mathfrak{B}^{\pm}$  is the union of the spherical focal surface of the spacelike and timelike part of $\gamma$, with the curves  $\mathfrak{B}^{\pm}(t_0,\mu)$, $-1<\mu<1$.

\end{proof}




\begin{thebibliography}{10}

\bibitem{bruce} {\sc Bruce, J. W., and Giblin, P. J.}
\newblock{\em Curves and Singularities: a geometrical introduction to
  singularity theory}.
Cambridge university press, 1992.

\bibitem{Fusho}
{\sc Fusho, T., and Izumiya, S.}
\newblock Lightlike surfaces of spacelike curves in de {S}itter 3-space.
\newblock {\em J. Geom. 88}, 1-2 (2008), 19--29.

\bibitem{izumiyadevelopable}
{\sc Izumiya, S.}
\newblock Generating families of developable surfaces in $\mathbb{R}^3$.
  \newblock{\em {H}okkaido {U}niversity preprint series in mathematics 512,} 2001.

\bibitem{izumiya}
{\sc Izumiya, S., Kikuchi, M., and Takahashi, M.}
\newblock Global properties of spacelike curves in {M}inkowski 3-space.
\newblock {\em J. Knot Theory Ramifications 15}, 7 (2006), 869--881.

\bibitem{izumiyapeisanosingularities}
{\sc Izumiya, S., Pei, D., and Sano, T.}
\newblock Singularities of hyperbolic gauss maps.
\newblock {\em Proceedings of the London mathematical Society 86}, 2 (2003),
  485--512.

\bibitem{kh}
{\sc Izumiya, S., Pei, D., and Sano, T.}
\newblock Horospherical surfaces of curves in hyperbolic space.
\newblock {\em Publ. Math. Debrecen 64}, 1-2 (2004), 1--13.

\bibitem{Izumiyapeisanotorri}
{\sc Izumiya, S., Pei, D.~H., Sano, T., and Torii, E.}
\newblock Evolutes of hyperbolic plane curves.
\newblock {\em Acta Math. Sin. (Engl. Ser.) 20}, 3 (2004), 543--550.

\bibitem{izumyiatakiyama}
{\sc Izumiya, S., and Takiyama, A.}
\newblock A timelike surface in {M}inkowski {$3$}-space which contains
  pseudocircles.
\newblock {\em Proc. Edinburgh Math. Soc. (2) 40}, 1 (1997), 127--136.

\bibitem{peisano}
{\sc Pei, D., and Sano, T.}
\newblock The focal developable and the binormal indicatrix of a nonlightlike
  curve in {M}inkowski 3-space.
\newblock {\em Tokyo J. Math. 23}, 1 (2000), 211--225.

\bibitem{pelletier}
{\sc Pelletier, F.}
\newblock Quelques propri{\'e}t{\'e}s g{\'e}om{\'e}triques des vari{\'e}t{\'e}s
  pseudo-riemanniennes singuli{\`e}res.
\newblock In {\em Annales de la facult{\'e} des sciences de Toulouse\/} (1995),
  vol.~4, Universit{\'e} Paul Sabatier, pp.~87--199.

\bibitem{Ratcliffe}
{\sc Ratcliffe, J.}
\newblock {\em Foundations of hyperbolic manifolds}, vol.~149.
\newblock Springer, 2006.

\bibitem{Saloomtari}
{\sc Saloom, A., and Tari, F.}
\newblock Curves in the {M}inkowski plane and their contact with
  pseudo-circles.
\newblock {\em Geom. Dedicata 159\/} (2012), 109--124.

\bibitem{stellergauss}
{\sc Steller, M.}
\newblock A gauss-bonnet formula for metrics with varying signature.
\newblock {\em Zeitschrift f{\"u}r Analysis und ihre Anwendungen 25}, 2 (2006),
  143--162.

\bibitem{farid}
{\sc Tari, F.}
\newblock Caustics of surfaces in the {M}inkowski 3-space.
\newblock {\em Q. J. Math. 63}, 1 (2012), 189--209.

\end{thebibliography}

{\small
\par\noindent
Ana Claudia Nabarro, Andrea de Jesus Sacramento,  Departamento de Matem\'{a}tica,
ICMC Universidade de S\~{a}o Paulo, Campus de S\~{a}o Carlos, Caixa Postal 668, CEP 13560-970, S\~{a}o Carlos-SP, Brazil
\par\noindent
e-mail:{\tt anaclana@icmc.usp.br }

\par\noindent e-mail:{\tt andreajs@icmc.usp.br}

}

\end{document}